\documentclass[a4paper,11pt]{amsart}

\usepackage{amssymb}
\usepackage{fullpage}
\usepackage{enumerate}
\usepackage[OT4]{fontenc}
\usepackage{tikz} 
\usepackage{float}
\usetikzlibrary{arrows}

\theoremstyle{definition}
\newtheorem{definition}{Definition}
\theoremstyle{plain}
\newtheorem{theorem}[definition]{Theorem}

\newtheorem{remark}[definition]{Remark}

\newtheorem{property}[definition]{Property}

\def\be{\begin{equation}}
\def\ee{\end{equation}}

\def\bp{\begin{property}}
\def\ep{\end{property}}

\def\bbP{{\mathbb P}}

\def\eee12{\frac{1}{\sqrt{r+\frac{1}{12}}}}

\def\ep{\varepsilon({\cal O}_{\bbP^2}(1), P_1,...P_r)}

\begin{document}

\title{\bf Some particular norm in the Sobolev space $H^{1}[a,b].$}

\author{Edward Tutaj}



\begin{abstract}

 {\small This paper is a continuation of the recent paper of the author, where a
 certain reproducing kernel Hilbert space $X_{\mathcal{S}}$ was constructed.
The norm in $X_{\mathcal{S}}$ is related to a certain generalized
isoperimetric inequality in ${\mathbb R^2}$. In the present paper
we give an alternative description of the space $X_{\mathcal{S}}$,
which appears to be a Sobolev space $H^{1}[a,b]$ with some special
norm.}
\end{abstract}

\maketitle

\section{Introduction}

This paper is a continuation of the recent paper of the author,
\cite{Tut}. In this paper we consider a Sobolev space $H^1[a,b]$.
We recall some necessary definitions and properties of these
spaces below, but at the beginning it is sufficient to know that
the elements of $H^1[a,b]$ are continuous functions with the
derivatives in $L_2[a,b]$. Then, if one will define an unitary
structure in $H^1[a,b]$, the formula for $\langle f,g\rangle$ must
necessarily contain the bilinear form ($f,g \in H^1[a,b]$)
\begin{equation}\label{formulawSobolewie1}
\int_a^b f'(x)\cdot g'(x)dx.
\end{equation}
Clearly the quadratic form generated by (\ref{formulawSobolewie1})
is only a seminorm in $H^1[a,b]$, hence to have a norm one must
add "something" to (\ref{formulawSobolewie1}), or consider a
smaller space. In (\ref{iloczyn skalarny 1}) we present a typical
example of this kind. In this paper we propose, a new formula to
define an inner product in $H^1[a,b]$ inspired by the very old
idea of the isoperimetric inequality in the plane. The paper
\cite{Tut} contains the construction of a certain reproducing
kernel Hilbert space  $X_\mathcal{S}$ (the necessary information
concerning these class of spaces are to be found in the next
section). Its kernel, given by the formula (\ref{rkhsK}), is not
very complicated. However the description of the function space
$X_\mathcal{S}$ is not satisfactory. The elements of this space
are the equivalent classes of pairs of convex sets relative to
some equivalence relation (formula (\ref{relacjakaro})) and then
it is not easy to look on the elements of $X_\mathcal{S}$ as on
the real function defined on some interval. But, on the other
hand, we have a concrete positively defined function $K$ (in our
situation the function $K(\varphi,\psi)$ given by the formula
(\ref{rkhsK})) so we may try to construct another, perhaps more
convenient, model of the space $\mathcal{H}(K)$ corresponding to
the considered kernel. Such a situation occurs frequently in the
theory of RKHS's  and has even his own name ({\it the
reconstruction problem,} cf \cite{Pau}). The aim of this paper is
to give such  a concrete description of the space $\mathcal{H}(K)$
for the kernel given by (\ref{rkhsK}). It will appear that in our
case the space $\mathcal{H}(K)$ is a Sobolev space $H^1[a,b]$
equipped with some special norm. The construction of "our"
$\mathcal{H}(K)$ is contained in the three theorems: Theorem
\ref{dodatnia okre\'slono\'s\'c}, Theorem \ref{aboutkernel} and
the Theorem \ref{completness}.

 In Theorem \ref{dodatnia okre\'slono\'s\'c}
we define a certain bilinear form $\langle,\rangle_i$
(\ref{iloczynizop}) and we prove that this form is positively
defined. In the proof we use the Wirtinger inequality
(\ref{nierWirtingera}) in a general version for the derivatives
from $L_2[a,b]$. This variant of the Wirtinger inequality  is to
be found in \cite{Har}.

In Theorem \ref{aboutkernel} we check the reproducing property of
the given kernel (\ref{rkhsK}). The proof is easy to understand,
but may be is heavy to read because of lengthy calculations.  The
kernel (\ref{rkhsK}), as far as we know, was not studied before.

In the last Theorem \ref{completness} we prove, that the
constructed space is an RKHS.  For this we must  check, that the
evaluation functionals are commonly bounded, and that the
constructed space is complete. This type of arguments is typical
in  examples of Sobolev spaces $H^1[a,b]$ \cite{Ber}.
    Summarizing,
the main result of this paper says, that the space
$H^{1}[-\frac{\pi}{2},\frac{\pi}{2}]$ equipped with the norm
$||\cdot||_i$ is a reproducing kernel Hilbert space corresponding
to the same kernel (\ref{rkhsK}) as the space $X_{\mathcal{S}}$
constructed in \cite{Tut}.

We present also one more construction of the considered space
$\mathcal{H}(K)$, called here the {\it sequence model}. This
sequence model, if one looks from strictly theoretical point of
view, does not bring new information, but gives  tools for proofs
and for numerical calculations. This model represents also a kind
of presentation of the density of polygons in the space of convex
sets.

From a number of known examples of the kernels in the space
$H^1[a,b]$, we present here one see formulas (\ref{Berlinet1}) and
(\ref{Berlinet2}) in order to compare it with the considered here
kernel (\ref{rkhsK}).

This paper is organized as follows. In the second we give a
summary of the paper \cite{Tut} necessary to understand the main
result of the present paper.  We recall also some information on
the Sobolev of the type $H^{1}[a,b]$ in the aim to see the
"particularity" of the norm we are going to construct. In the
third section we define some special (particular) norm
$||\cdot||_i$ in the space $H^{1}[-\frac{\pi}{2},\frac{\pi}{2}]$.
In the last section we present the construction of the {\it
sequence model}.

 \section{About the space of generalized convex sets} In this subsection
 we recall the construction of the space $X_\mathcal{S}$ from \cite{Tut}. Let
 $\mathcal{S}$ denote the cone of convex, compact and centrally
 symmetric sets in the plane $\mathbb R^2$. In this cone we
 consider the {\it Minkowski addition} and the scalar multiplication by
 non-negative real numbers. In the product $\mathcal{S}\times
 \mathcal{S}$ we consider the equivalence relation
 \begin{equation}\label{relacjakaro}
 (U,V)\diamondsuit(P,Q)\Longleftrightarrow U+Q=V+P
 \end{equation}
 It is known, that the quotient $\mathcal{S}\times
 \mathcal{S}/\diamondsuit$ is a real vector space (for details see e.g. \cite{Rad}),
 and the
 space $X_{\mathcal{S}}$ mentioned above is the completion of the
 quotient $\mathcal{S}\times \mathcal{S}/\diamondsuit$ with
 respect to the so-called {\it isoperimetric norm}, which is
 constructed in a few steps. At first we prove, that the two
 dimensional Lebesgue measure $m:\mathcal{S}\longrightarrow
 [0,\infty)$ can be extended in a unique way to the polynomial
 of the second degree $m^{*}:\mathcal{S}\times
 \mathcal{S}/\diamondsuit \longrightarrow \mathbb R$ and the functional
 $o:\mathcal{S}\longrightarrow \mathbb R$ where $o(U)$ is a perimeter of U,
 can be extended to a linear functional
 on the whole $\mathcal{S}\times \mathcal{S}/\diamondsuit$. Next
 we
 prove that for this extended measure $m^{*}$ the {\it generalized
 isoperimetric inequality} holds, which means that for each
 $[U,V]\in \mathcal{S}\times \mathcal{S}/\diamondsuit$ we have
 \begin{equation}\label{nierownoscizoperymetryczna}
 (o(U)-o(V))^2-4\pi\cdot m([U,V])\geq 0,
 \end{equation}
 where $o(U)-o(V)$ is the perimeter of the pair $[U,V]$ and (here and in the sequel)
 we write $m([U,V])$ instead of $m^{*}([U,V])$. Using this fact it was proved, that
  the formula
  \begin{equation}\label{normaizoper}
  ||[U,V]||^2_i=2\cdot(o(U)-o(V))^2-4\pi\cdot m([U,V])
  \end{equation}
  is a unitary norm in $\mathcal{S}\times
  \mathcal{S}/\diamondsuit$. Finally the space $X_\mathcal{S}$ was
  defined as the completion of the unitary space
  $(\mathcal{S}\times \mathcal{S}/\diamondsuit; ||\cdot||_i)$.

 \subsection{About the RKHS}In the last section of \cite{Tut}
 we prove that the space
$X_{\mathcal{S}}$ is a {\it reproducing kernel Hilbert space}
(RKHS for short).
 Although the reproducing property was discovered by Zaremba more
than hundred years ago (1906), the first systematic lecture is due
to Aronszajn in 1950 (see \cite{Aro}). The necessary information
concerning RKHS are to be found in the book of Berlinet and
Thomas-Agnant (\cite{Ber}) or in the recent book of Szafraniec
(\cite{Szaf}). The RKHS are function spaces, so at first one must
recognize the elements of $\mathcal{S}\times
\mathcal{S}/\diamondsuit$ as real functions on the interval
$\Delta=[-\frac{\pi}{2},\frac{\pi}{2}]$. This may be done as
follows. We consider the functions
\begin{equation}\label{diangles}
I_{\psi}:[-\frac{\pi}{2},\frac{\pi}{2}]\ni \varphi\longrightarrow
\sin|\varphi-\psi|,
\end{equation}
where $\varphi,\psi \in [-\frac{\pi}{2},\frac{\pi}{2}].$ These
functions $I_\psi$ are called {\it diangles}. Each diangle defines
a line in $\mathbb R^2$ ($\mathbb R \cdot (\cos \psi,\sin\psi))$
which will be denoted also by $I_\psi$ and for the set $U\in
\mathcal{S}$ the number $\overline{U}(I_\psi)$ denotes the {\it
width} of the set $U$ with respect to the line $I_\psi$. The
correspondence
\begin{equation}\label{szerokosc}
 [U,V]\rightarrow f_{[U,V]} = \overline{U}(I_\varphi) -
  \overline{V}(I_{\varphi}),
\end{equation}
allows us  to see the equivalent classes $[U,V]$ as functions on
the interval $\Delta$ (the details in \cite{Tut}). Now we are
ready to formulate the main result from \cite{Tut}.

\begin{theorem}\label{rkhsX}
The space $X_\mathcal{S}$ with the norm (\ref{normaizoper}) is a
reproducing kernel Hilbert space with the kernel $K:\Delta\times
\Delta \longrightarrow \mathbb R$, where
\begin{equation}\label{rkhsK}
K(\varphi,\psi)= 2 - \frac{\pi}{2}\sin|\varphi-\psi|.
\end{equation}
\end{theorem}

The elements of $X_\mathcal{S}$ are, roughly speaking, the
"differences" of convex sets (differences with respect to the
Minkowski addition). This makes it possible to understand their
geometrical character, but on the other hand, given a function
$f:\Delta\longrightarrow \mathbb R$ it is difficult to prove (or
disprove), that $f\in X_\mathcal{S}$. In this paper we will solve
this problem. More exactly we will prove, that $X_\mathcal{S}$ is
a certain Sobolev space.

 \subsection{About the Sobolev spaces.} We shall start by recalling some commonly
  known definitions and
theorems concerning the Sobolev spaces.

\begin{definition}\label{definicjaabsciaglasci}

Let  $\Delta =[a,b]\subset \mathbb R$ be a compact interval. A
function $f:\Delta \longrightarrow \mathbb R$ is said to be {\it
absolutely continuous} on $\Delta$ if for every $\epsilon > 0$
there exists $\delta >0$ such that

\begin{equation}
\sum_{k=1}^{p}|f(b_k)-f(a_k)|\leq \epsilon
\end{equation}
for every finite number of nonoverlappig intervals $(a_k,b_k)$,
$k=1,2,...,p$ with $[a_k,b_k]\subset \Delta$ and

\begin{equation}
\sum_{k=1}^{p}(b_k-a_k)\leq \delta.
\end{equation}
\end{definition}

Let $AC_1(\Delta)$ denote the vector space of all absolutely
continuous functions on $\Delta$. It is known, that each $f\in
AC_1(\Delta)$ is differentiable almost everywhere and its
derivative $f'$ is a Lebesgue integrable function (i.e. $f'\in
L_1(\Delta)$). We will consider a subspace $H^{1}(\Delta)\subset
AC_1(\Delta)$ claiming that

\begin{equation}\label{definicjaH1}
f\in H^{1}(\Delta) \Longleftrightarrow f\in AC_1(\Delta) \wedge
f'\in L_2(\Delta).
\end{equation}

It is also known (see e.g. \cite{Pau}), that one may consider an
inner product $\langle,\rangle$ in $H^{1}(\Delta)$ given by the
formula (\ref{iloczyn skalarny 1}) for $f,g \in H^{1}(\Delta)$:

\begin{equation}\label{iloczyn skalarny 1}
\langle f,g\rangle=   \int_{a}^{b}(f\cdot g + f'\cdot g')
\end{equation}
which is well defined since $f',g' \in L_2(\Delta)$. One may
check, that $H^{1}(\Delta)$, equipped with the above inner product
(\ref{iloczyn skalarny 1}), is a Sobolev space with the
reproducing kernel $K(x,y)$ where
\begin{equation}\label{Berlinet1}
K(x,y)=\frac{\cosh(x-a)\cosh(b-y)}{\sinh(b-a)},\\\ a\leq x\leq
y\leq b
\end{equation}

\begin{equation}\label{Berlinet2}
K(x,y)= \frac{\cosh(x-a)\cosh(b-y)}{\sinh(b-a)},\\\ a \leq y\leq x
\leq b
\end{equation}

Some other examples of reproducing kernels in Sobolev spaces are
to be found in \cite{Ber} or in \cite{Pau}.

In the present paper we  construct another inner product
$\langle,\rangle_i$ defined in a subspace
$H^{1}_{0}(\Delta)\subset H^{1}(\Delta)$, related to another
reproducing kernel $K_i(x,y),$ which will be defined in the next
section.

\section{An inner product in the Sobolev  space $H_0^{1}(\Delta)$}

Now we fix $\Delta=[-\frac{\pi}{2},\frac{\pi}{2}]$ and we consider
the subspace $H^{1}_{0}\subset H^{1}(\Delta)$ claiming that
$f(-\frac{\pi}{2})=f(\frac{\pi}{2})$. This means, that
$H^1_0(\Delta)$ may be identified with the space of periodic
functions with analogous properties (i.e. absolutely continuous
and with the derivatives in $L_2(\Delta)$. For $f,g \in
H_0^{1}(\Delta)$ we set

\begin{equation}\label{iloczynizop}
\langle f,g\rangle_i = \frac{1}{\pi^2}[2\cdot
(\int_{-\frac{\pi}{2}}^{\frac{\pi}{2}}f)\cdot
(\int_{-\frac{\pi}{2}}^{\frac{\pi}{2}} g) - \pi
\cdot\int_{-\frac{\pi}{2}}^{\frac{\pi}{2}} (f\cdot g - f'\cdot
g')].
\end{equation}

Here, and below we will often write $\int_a^b f$ or $\int_a^b
f(x)$ instead of $\int_a^b f(x)dx$. We will frequently also write
$\int f$ instead of $\int_a^b f$ in the case when $a=
-\frac{\pi}{2}$ and $b=\frac{\pi}{2}$.

It is clear, that the formula (\ref{iloczynizop}) is well defined
(because of $f',g' \in L_2(\Delta)$) and defines a bilinear  form
in $H_0^{1}(\Delta)$. The coefficient $\frac{1}{\pi^2}$ is clearly
without importance and is chosen  to have $\langle{\bf 1},{\bf
1}\rangle_i = 1$ (where ${\bf 1}$ denotes the constant (and equal
1) function). Hence, to have a norm related to the form
(\ref{iloczynizop}), it remains to show that $\langle,\rangle_i$
is positively defined. The proof of this positivity is similar in
fact to the known proof of the isoperimetric inequality based on
the Wirtinger inequality.

Let us start by recalling a variant of the Wirtinger inequality
(for the proof and some other formulations see
\cite{Har},\cite{Tre}).

\begin{theorem}\label{nierWirtingera}
Let $f$ be a continuous and periodic function with the period
$\pi$. Let $f'\in L_2(\Delta)$ and let
\begin{equation}\label{srednia}
\overline{f} =
\frac{1}{\pi}(\int_{-\frac{\pi}{2}}^{\frac{\pi}{2}}f),
\end{equation}
Then
\begin{equation}\label{nierWirtbis}
\int_{-\frac{\pi}{2}}^{\frac{\pi}{2}}(f')^2 \geq
\int_{-\frac{\pi}{2}}^{\frac{\pi}{2}}(f-\overline{f})^2,
\end{equation}
 and
equality holds if and only if $f=a \cos(t) + b\sin(t)$.

\end{theorem}

Now we have the following

\begin{theorem}\label{dodatnia okre\'slono\'s\'c}
Let $f\in H_0^{1}(\Delta)$. Then

\begin{equation}\label{ni}
\langle f,f\rangle_i=2(\int_{-\pi/2}^{\pi/2}f)^2 - \pi
\int_{-\pi/2}^{\pi/2}(f)^2-(f')^2\geq 0.
\end{equation}
and the equality holds only when $f=0$. In other words, the form
(\ref{iloczynizop}) is positively defined.
\end{theorem}

\begin{proof}

 Let us set, as above
\begin{equation}\label{srednia}
 \overline{f} =
\frac{1}{\pi}(\int_{-\frac{\pi}{2}}^{\frac{\pi}{2}}f),
\end{equation}

 and let

\begin{equation}\label{deficytE}
E(f)=(\int_{-\frac{\pi}{2}}^{\frac{\pi}{2}}f)^2 - \pi
\int_{-\frac{\pi}{2}}^{\frac{\pi}{2}}((f)^2-(f')^2)
\end{equation}

We want to show, that $E(f)\geq 0$. Consider the difference

\begin{equation}\label{errorF}
F(f)=\int_{-\frac{\pi}{2}}^{\frac{\pi}{2}}(f')^2 -
\int_{-\frac{\pi}{2}}^{\frac{\pi}{2}}(f-\overline{f})^2.
\end{equation}

 The Wirtinger inequality says precisely that $F(f)\geq 0$.
This implies that
\begin{equation}\label{errorE}
E(f)=\int_{-\frac{\pi}{2}}^{\frac{\pi}{2}}f)^2 - \pi
\int_{-\frac{\pi}{2}}^{\frac{\pi}{2}}(f)^2-(f')^2\geq 0.
\end{equation}

Indeed $$0\leq F(f)=\int_{-\frac{\pi}{2}}^{\frac{\pi}{2}}(f')^2
-\int_{-\frac{\pi}{2}}^{\frac{\pi}{2}}(f)^2 + 2
\int_{-\frac{\pi}{2}}^{\frac{\pi}{2}}f \cdot (\frac{1}{\pi}\cdot
\int_{-\frac{\pi}{2}}^{\frac{\pi}{2}}f) -\pi\cdot
(\frac{1}{\pi^2}\cdot
(\int_{-\frac{\pi}{2}}^{\frac{\pi}{2}}f)^2)$$
$$= \int_{-\frac{\pi}{2}}^{\frac{\pi}{2}}(f')^2 -\int_{-\frac{\pi}{2}}^{\frac{\pi}{2}}(f)^2
+\frac{1}{\pi}(\int_{-\frac{\pi}{2}}^{\frac{\pi}{2}}f)^2,$$ and
this is equivalent to  the desired inequality $E(f)\geq 0$.  Then
$\langle f,f\rangle_i\geq 0$.

If $\langle f,f\rangle_i =0$ then $$\langle f,f\rangle_i =
(\int_{-\frac{\pi}{2}}^{\frac{\pi}{2}}f)^2 + E(f) =0$$and
 hence $\overline{f}=0 $ and $E(f)=0$. Thus
 $$\int_{-\frac{\pi}{2}}^{\frac{\pi}{2}}((f)^2-(f')^2)=0$$and
by (\ref{ni}) $F(f)=0$, which means that the equality holds in the
Wirtinger inequality and additionally $\overline{f}=0$.
 We know that in such
 a case $f=a \cos(t) +b\sin(t)$ and moreover $\overline{f}=0$.
 Since $f(-\frac{\pi}{2})=f(\frac{\pi}{2})$ then
$b=0$ and since $\overline{f}=0$ then $a=0$. Hence $f=0$.

In consequence we have proved that the  $H^1_{0}(\Delta)$ with the
inner product (\ref{iloczynizop}) is an unitary space. The norm
induced by (\ref{iloczynizop}) is

\begin{equation}\label{normi}
||f||_i^2=\frac{1}{\pi^2}\cdot(2\cdot(\int_{-\pi/2}^{\pi/2}f)^2
-\pi\int_{-\pi/2}^{\pi/2}((f)^2-(f')^2)=
 \frac{1}{\pi^2}(o^2(f) + E(f)),
\end{equation}
where $$o^2(f)=(\int_{-\frac{\pi}{2}}^{\frac{\pi}{2}}f)^2.$$ This
ends the proof of  Theorem \ref{dodatnia okre\'slono\'s\'c}.

\end{proof}

We will prove now the reproducing property.

\begin{theorem} \label{aboutkernel}
The function $K:[-\frac{\pi}{2},\frac{\pi}{2}]\times
[-\frac{\pi}{2},\frac{\pi}{2}]\longrightarrow \mathbb R$ defined
by
\begin{equation}\label{kernel}
K_i(x,y)= 2-\frac{\pi}{2}\cdot \sin|x-y|,
\end{equation}
is a reproducing kernel in the unitary space $H^1_{0}(\Delta)$.
This function will be called {\it the isoperimetric kernel}.
\end{theorem}

\begin{proof}
 Let $k_y(x)=2-\frac{\pi}{2}\cdot \sin|x-y|,$ (i.e. $k_y= K_i(\cdot,y)$
 is a kernel function).
  We will check (this is the so - called reproduction property) that
for each function $f\in H^1_0(\Delta)$ and for each $y\in
[-\frac{\pi}{2},\frac{\pi}{2}]$ we have:
\begin{equation}\label{wlasnosc reprodukowania}
f(y)=\langle f,k_y\rangle_i
\end{equation}

Let us fix $y$ and $f$ as above. We must verify that
\begin{equation}\label{formula na f(y)}
f(y) = \frac{1}{\pi^2}(2\cdot(
\int_{-\frac{\pi}{2}}^{\frac{\pi}{2}}f)(\int_{-\frac{\pi}{2}}^{\frac{\pi}{2}}k_y)-
\pi\cdot\int_{-\frac{\pi}{2}}^{\frac{\pi}{2}}(f\cdot k_y - f'\cdot
k'_y)).
\end{equation}

Let us denote
$$A=2\cdot(\int_{-\frac{\pi}{2}}^{\frac{\pi}{2}}f)(\int_{-\frac{\pi}{2}}^{\frac{\pi}{2}}k_y),$$

$$B=\int_{-\frac{\pi}{2}}^{\frac{\pi}{2}}f\cdot k_y,$$

$$C=\int_{-\frac{\pi}{2}}^{\frac{\pi}{2}}f'\cdot k'_y.$$

We must check that $$f(y)=\frac{1}{\pi^2}(A-\pi B +\pi C).$$

 Let us recall, that in the  calculations which will be done below, we will frequently write
 $\int h$ instead of $\int_{-\frac{\pi}{2}}^{\frac{\pi}{2}}h(x)dx.$

We have
\begin{equation}\label{wyliczenie A}
A= 2\cdot(\int_{-\frac{\pi}{2}}^{\frac{\pi}{2}}f)\cdot
\int_{-\frac{\pi}{2}}^{\frac{\pi}{2}}(2-\frac{\pi}{2}\sin|x-y|)=2\cdot(\int_{-\frac{\pi}{2}}^{\frac{\pi}{2}}f)\cdot
(2\pi - \pi)=2\pi\cdot (\int_{-\frac{\pi}{2}}^{\frac{\pi}{2}}f).
\end{equation}

Now we will compute $B$.

\begin{equation}\label{wyliczenie B}
 B= \int_{-\frac{\pi}{2}}^{\frac{\pi}{2}} f\cdot(2-\frac{\pi}{2}\cdot \sin|x-y|)
  =2\int_{-\frac{\pi}{2}}^{\frac{\pi}{2}} f -
\frac{\pi}{2}\cdot \int f\cdot \sin|x-y| =$$ $$ 2\int f -
\frac{\pi}{2}\cdot \int f\cdot \sin|x-y| = 2\int f -
\frac{\pi}{2}\cdot(\int_{-\frac{\pi}{2}}^{y}f\cdot \sin(y-x) +
\int_{y}^{\frac{\pi}{2}}f\cdot \sin(x-y))
\end{equation}

The case of $C$ is more complicated. We have
$$C=\int f'\cdot k'_y = \int
f'\cdot(-\frac{\pi}{2}\cdot\sin|x-y|)=$$
$$=\int_{-\frac{\pi}{2}}^{y}(-\frac{\pi}{2})f'\cdot(-\cos(x-y))+\int^{\frac{\pi}{2}}_{y}(-\frac{\pi}{2})
f'\cos(x-y)=$$
$$=\frac{\pi}{2}\int_{-\frac{\pi}{2}}^{y} f'\cos(x-y) -
\frac{\pi}{2}\int^{\frac{\pi}{2}}_{y} f'\cos(x-y).$$

Now we apply twice the integration by parts and we obtain that the
above is equal
$$=\frac{\pi}{2}[f\cos(x-y)|_{-\frac{\pi}{2}}^y
+ \int_{-\frac{\pi}{2}}^{y}f\sin(x-y) -
f\cos(x-y)|^{\frac{\pi}{2}}_y -
\int^{\frac{\pi}{2}}_{y}f\sin(x-y)]=$$
$$=\frac{\pi}{2}[1\cdot f(y)-f(\frac{-\pi}{2})\cos(-\frac{\pi}{2}-y)-f(\frac{\pi}{2})\cos(\frac{\pi}{2}-y)+
1\cdot f(y) - (\int_{\frac{-\pi}{2}}^{y}f(-\sin(x-y)) +
\int^{\frac{\pi}{2}}_{y}f(-\sin(x-y)))].$$

Since $f(-\frac{\pi}{2})=f(\frac{\pi}{2})$ then we obtain

$$
\frac{\pi}{2}[2f(y)-f(\frac{\pi}{2})(\cos(\frac{\pi}{2}+y)+\cos(\frac{\pi}{2}-y))-
\int_{-\frac{\pi}{2}}^{\frac{\pi}{2}}f\cdot \sin|x-y|].$$

But $\cos(\frac{\pi}{2}+y)=-\sin(y)$ and
$\cos(\frac{\pi}{2}-y)=\sin(y)$ so finally we have
 $$C=\frac{\pi}{2}[2f(y)-\int_{-\frac{\pi}{2}}^{\frac{\pi}{2}}f\cdot
 \sin|x-y|].$$

Hence $$\langle f,k_y\rangle_i=\frac{1}{\pi^2}[A-\pi B+\pi C]=
\frac{1}{\pi^2}[2\pi(\int f)-\pi(2(\int f)-\frac{\pi}{2}f\cdot
\sin|x-y|)+\pi\frac{\pi}{2}(2f(y)-\int_{-\frac{\pi}{2}}^{\frac{\pi}{2}}f\cdot
 \sin|x-y|)]$$
$$=\frac{1}{\pi^2}[2\pi(\int f) -2\pi(\int f) +\frac{\pi^2}{2}\int
f\sin|x-y| + \pi^2 f(y) - \frac{\pi^2}{2}\int f\sin|x-y|]= $$
$$=\frac{1}{\pi^2}\cdot{\pi^2} f(y) = f(y).$$

This ends the proof of the Theorem \ref{aboutkernel}
\end{proof}

It remains to prove now the main result of this section, which
says that $H^{1}_{0}(\Delta)$ with the norm (\ref{normi}) is an
RKHS.

\begin{theorem}\label{completness}

With the notations as above, the space $H_{0}(\Delta)$ is a
reproducing kernel Hilbert space and its kernel is given by
\begin{equation} K_i(x,y)=2-\frac{\pi}{2}\sin|x-y|.
\end{equation}
\end{theorem}

\begin{proof}

The argumentation is similar for example to that in \cite{Pau}.
First we check that the evaluation functionals $e_x(f)=f(x)$ are
bounded with respect to the isoperimetric norm. Since
$H_0(\Delta)=H_{00}+\mathbb {R}\cdot \mathbb {\bf 1}$  then it is
sufficient to prove, that each $e_x$ is bounded on the subspace
$H_{00}\subset H^1_{0}$, composed of those functions which vanish
at $0$, (i.e. we may additionally assume, that $f(0)=0)$.

Let us suppose that $f\in H^1_{0}(\Delta)$. Let $x+h,x\in
[-\frac{\pi}{2},\frac{\pi}{2}]$. By the reproducing property (just
proved) we have:

$$f(x+h)-f(x) = \langle f,k_{x+h}\rangle_i - \langle f,k_x\rangle_i = \langle f,k_{x+h}-k_x\rangle_i,$$
or
$$k_{x+h}(t)-k_x(t) =
(2-\frac{\pi}{2}\sin|t-(x+h)|)-(2-\frac{\pi}{2}\sin|t-x|)=
\frac{\pi}{2}(I_x(t)-I_{x+h}(t)),$$ where $I_{\psi}$ is a diangle
(see (\ref{diangles}),\cite{Tut}). More exactly $I_\psi$ is a
function given by the formula
$I_\psi(\varphi)=\sin|\varphi-\psi|$. Hence we may write (using
Schwarz inequality)

\begin{equation}\label{holder1}
|f(x+h)-f(x)| = |\frac{\pi}{2}\langle f,I_x - I_{x
+h}\rangle_i|\leq \frac{\pi}{2}||f||_i\cdot||I_x-I_{x+h}||_{i}.
\end{equation}
 Now we may compute
the isoperimetric norm of the difference of diangles
$I_x-I_{x+h}$. Namely, using the formula for the isoperimetric
norm from \cite{Tut} we have:

\begin{equation} \label{holder2}
||I_{x+h}-I_x||^{2}_{i}= \frac{1}{4\pi^2}[2(o(I_x)-o(I_{x+h}))^2 -
4\pi m ([I_{x+h},I_x])]=\frac{1}{4\pi^2}\cdot 4\pi\cdot 4 \cdot
1\cdot 1 \cdot \sin|h| = \frac{4}{\pi}\sin|h|.
\end{equation}

However, if we want to have an independent proof of  Theorem
\ref{completness}, we must  compute the norm $||I_{x+h}-I_x||^2$
using  formula (\ref{ni}). A direct calculation of the definite
integrals which are in (\ref{ni}) is perhaps not to difficult, but
is rather time-absorbing. Thankfully, we have the reproducing
property (just proved). Hence
$$||I_\varphi - I_\psi||^2= \langle I_\varphi - I_\psi,I_\varphi -
I_\psi\rangle_i = \frac{2}{\pi}\langle I_\varphi -
I_\varphi,\frac{\pi}{2}I_\psi - 2, 2-\frac{\pi}{2}I_\psi\rangle_i=
\frac{2}{\pi}\langle I_\varphi - I_\psi,-k_\psi+k\varphi\rangle_i
=$$

$$=\frac{2}{\pi}(\langle I_\varphi,-k_\varphi\rangle_i + \langle I_\psi,k_\varphi\rangle_i+
\langle I_\varphi,k_\psi\rangle_i -\langle
I_\psi,k_\psi\rangle_i)= \frac{2}{\pi}(-0 + I_\varphi(\psi) +
I_\psi(\varphi) + 0) =$$
$$=\frac{4}{\pi}\sin|\varphi-\psi|.$$ Setting $h=\varphi-\psi$ we
obtain the desired equality. In consequence
\begin{equation}\label{Holder1}
||I_{x+h}-I_x||_i \leq \sqrt{\frac{4}{\pi}}\sqrt{|h|},
\end{equation}
and hence

\begin{equation} \label{holder3}
|f(x+h)-f(x)|\leq \sqrt{\pi}\cdot \sqrt{|h|}(||f||_i).
\end{equation}

Setting $x=0$ and taking $f$ such that $f(0)=0$ we obtain
$$|f(h)|\leq \sqrt{\pi}\cdot {\sqrt{|h|}}(||f||_i).$$

Hence all evaluation functionals $(e_t)_{t\in \Delta}$ are
bounded. Moreover the norm of uniform convergence on $\Delta$ is
weaker than the convergence with respect to the isoperimetric
norm.

To end the proof of  Theorem \ref{completness} it remains to show,
that $H^{1}_{0}(\Delta)$ with the isoperimetric norm is complete.
Take then a sequence $(f_n)$, which is a Cauchy sequence with
respect to the isoperimetric norm. Hence
$${\pi^2}||f_n-f_m||^2 = o^2(f_n-f_m) + E(f_n-f_m).$$ By our
assumption both sequences $o^2(f_n-f_m)$ and $ E(f_n-f_m)$ tend to
0 (when $m,n \rightarrow \infty$). But $$E(f_n-f_m)= o^2(f_n-f_m)
-\pi (\int(f_n-f_m)^2 - \int(f'_n-f'_m)^2).$$ Since $o^2(f_n-f_m)$
tends to $0$ then $(\int(f_n-f_m)^2 -\int(f'_n-f'_m)^2)$ tends to
$0$. But, by the remark made above, the sequence $(f_n)$ being a
Cauchy sequence with respect to the isoperimetric norm is also a
Cauchy sequence with respect to the norm of uniform convergence on
the compact interval $[-\frac{\pi}{2}, \frac{\pi}{2}]$. Hence
$\int(f_n-f_m)^2$ tends to $0$ and in consequence
$\int(f'_n-f'_m)^2)$ tends to $0$. This means, in particular, that
$(f'_n)$ is a Cauchy sequence in $L^2[-\frac{\pi}{2},
\frac{\pi}{2}]$, then $f'_n$ tends to a function $h$ from
$L^2[-\frac{\pi}{2}, \frac{\pi}{2}]$ with respect to the $L^2$
norm. But $f_n$ tends uniformly to a continuous function $f$.
Since
$$f(x) = \lim_{n}f_n(x) = \lim_n\int_{-\frac{\pi}{2}}^{x}
f'_n(t)dt =\int_{-\frac{\pi}{2}}^{x} h(t)dt,$$ then $f$ is
absolutely continuous and since
$f_n(-\frac{\pi}{2})=f_n(\frac{\pi}{2})$ then also
$f(-\frac{\pi}{2})=f(\frac{\pi}{2})$. Thus $f\in H^1_{0}(\Delta)$
and $||f_n-f||_i$ tends to $0$. Hence $ H^1_{0}(\Delta)$ is
complete with respect to the isoperimetric norm.

\end{proof}

We will end this chapter by a number of remarks.

\begin{remark}\label{remark6}

{\rm The last theorem asserts, that the space $H_0(\Delta)$ with
the norm
\begin{equation}\label{norma w XS}
||f||_i^2=\frac{1}{\pi^2}\cdot(2\cdot(\int_{-\pi/2}^{\pi/2}f)^2
-\pi(\int_{-\pi/2}^{\pi/2}(f)^2-(f')^2)
 =\frac{1}{\pi^2}(o^2(f) + E(f))
 \end{equation}
is isometric to the space $X_\mathcal{S}$  of generalized convex
sets with the norm
\begin{equation} \label{norm w XS}
||[U,V]||^2 =\frac{1}{4\pi^2}[2o^2([U,V])- 4\pi\cdot m([U,V]),
\end{equation}
where $o([U,V]) = o(U)-o(V)$ is the generalized perimeter and

\begin{equation}\label{miararoznicy}
m([U,V])=2m(U)+2m(V)-m(U+V)
\end{equation}
is the generalized Lebesgue measure in the plane. This is true
since both spaces have the same (isoperimetric) kernel. The
details concerning the formulas (\ref{norm w XS}) and
(\ref{miararoznicy}) are described in (\cite{Tut}). The isometry
between $X_\mathcal{S}$ and $H^1_0(\Delta)$ is established by the
correspondence
  $$\mathcal{S}\times \mathcal{S} \ni [U,V]\rightarrow f_{[U,V]}\in H^1_o(\Delta)$$ where
  $$f_{[U,V})(\varphi) = \overline{U}(I_\varphi) -
  \overline{V}(I_{\varphi}),$$and where $\overline{U}(I_\varphi)$
  denotes the width of the set $U$ with respect to the straight
  line generated by the diangle $I_\varphi$, (see \cite{Tut}).

This means, that the term $\int f = \int f_{[U,V]}$ is (because of
the isometry) the perimeter of $[U,V]$ (in particular of $U$) and
the term $$\int_{-\pi/2}^{\pi/2}(f_{[U,V]})^2-(f_{[U,V]}')^2=
m([U,V])$$ equals to the measure of $[U,V]$. This fact is known
for a long time as the so-called Cauchy formula (true not only for
centrally symmetric, but for all convex and compact  sets). Let us
notice that  (\ref{completness}) may be considered as an extension
of the Cauchy formulas for generalized measure.}

\end{remark}

\begin{remark}

{\rm Let us come back to the inequality (\ref{holder3}) proved
above, and saying that for each $f\in H^1_0(\Delta)$, for each
$x\in \Delta$ and $h\rangle0$ the following inequality holds:
\begin{equation}
|f(x+h)-f(x)|\leq \sqrt{\pi}\cdot ||f||_i\sqrt{h}.
\end{equation}
Setting $x+h=x_2$ , $x=x_1$ and $\sqrt{\pi} ||f||_i = C(f)$ we
obtain the evaluation

$$|f(x_2)-f(x_1)|\leq C(f)\sqrt{|x_2-x_1|} =
C(f)|x_2-x_1|^{\frac{1}{2}}. $$

This means, that each function $f\in H^1_0(\Delta)$ (each
$[U,V]\in X_\mathcal{S}$ ) satisfies the H$\ddot{o}$lder condition
with the exponent $\frac{1}{2}$. Hence the function
$$u(x) = |x|^{\frac{1}{3}}\notin X_\mathcal{S}$$and the diangles
are examples, that the exponent $\frac{1}{2}$ cannot be improved.
This is the known H$\ddot{o}$lder condition in Sobolev spaces and
we see that it is not too difficult to prove, when we work in the
representation $H^1_0(\Delta)$. On the other hand, this is not
easy to observe, if we work with generalized convex sets in
$X_\mathcal{S}$.}

\end{remark}
\begin{remark}

{\rm The third remark concerns the coefficient 2 in the formula
$$||[U,V]||^2 =\frac{1}{4\pi^2}[2o^2([U,V])-
4\pi\cdot m([U,V]),$$or in the formula (\ref{ni}). The generalized
isoperimetric inequality
$$o^2([U,V])-4\pi m([U,V])\geq 0,$$ or more exactly the inequality
 $$o^2(x)-4\pi m(x)\geq 0$$
for $x\in X_\mathcal{S}$ gives only a seminorm on $X_\mathcal{S}$
with one-dimensional kernel ${\mathbb R}\cdot {\bf 1}$. Thus as it
was remarked in the Introduction,  to have the "true" norm one
must add "something" non-vanishing on the line ${\mathbb R}\cdot
{\bf 1}$. In this (and the previous \cite{Tut} paper, this added
term is the seminorm $o^2([U,V])$ (more exactly $o^2(x)$ - a
square of the linear functional - for $x\in X_\mathcal{S}$). It is
clear, that one may consider the family of norms of the form
$$||[U,V]||^2 =\frac{1}{4\pi^2}[\theta \cdot o^2([U,V])-
4\pi\cdot m([U,V]),$$ where $\theta > 1$.

The corresponding reproducing kernel has the following form:
$$K_\theta(\varphi,\psi)=
\theta-\frac{\pi}{2}\sin|\varphi-\psi|,$$where $\theta$ is a
greater than 1. For $\theta=2$ we obtain the kernel $K_i$. Let us
observe additionally, that for $\theta=1$ we obtain an interesting
(since canonical) kernel
$$K_1(x,y) = 1-\frac{\pi}{2}\sin|x-y|,$$ corresponding to the subspace of
$X_o\subset X_\mathcal{S}$, which is the kernel of the linear
functional (perimeter) $x\longrightarrow o(x)$. }

\end{remark}

\section{A sequence model}

\vspace{5mm}

The construction of the RKHS space $H^1_0(\Delta)$ presented above
is frequently  called {\it from space to kernel} (see e.g.
\cite{Szaf}). In this section we will construct the same function
space - i.e. $X_\mathcal{S}$ - but in a way, which is called {\it
from kernel to space}. The aim for which we  repeat this commonly
known construction is that it gives a possibility to see a kind of
finite dimensional version of the generalized isoperimetric
inequality. As usual, in this model the space $X_\mathcal{S}$ will
appear to be the completion of a certain function space
$\mathcal{F}$ with respect to a suitable inner product (suitable
norm), corresponding to a given isoperimetric kernel. The idea we
will describe below in details, was used in (\cite{Tut}) in many
places, but in an implicit form.

\subsection{Definition of a certain sequence space.}

Let, as above $\Delta = [-\pi/2,\pi/2]$ and let us consider the
family of diangles, i.e. the family of functions
$(I_\varphi)_{\varphi\in \Delta}$, where
\begin{equation}\label{diangli}
I_{\varphi}:\Delta\ni \psi \longrightarrow
I_{\varphi}(\psi)=\sin|\varphi-\psi|.
\end{equation}

 Let $\mathcal{F}$ denote the subspace of the vector space ${\mathbb
 R}^{\Delta}$,
generated by the constant function $\mathbf{1}$ and by the
diangles i.e.

\begin{equation} \label{spandiangli}
\mathcal{F}=\left\{x_o \cdot \mathbf{1} + \sum_{i=1}^{k} x_i \cdot
I_{\varphi_i} : x_i\in \mathbb R, k=1,2,...\right\}.
\end{equation}

Let us observe, that the space $\mathcal{F}$ is exactly the linear
space generated by the kernel functions $k_\psi=
2\cdot{\bf{1}}-\frac{\pi}{2}I_\psi$ of the considered kernel
$K_i(\varphi,\psi)= 2\cdot {\bf
1}-\frac{\pi}{2}\sin|\varphi-\psi|.$ Now we will define a certain
bilinear form $\langle,\rangle_i$ in $\mathcal{F}$. Take  two
elements $x,y \in \mathcal{F}$ where
$$x=x_o\cdot \mathbf{1} + \sum_{i=1}^{k} x_i
\cdot I_{\varphi_i},$$ and $$y=y_o\cdot \mathbf{1} +
\sum_{j=1}^{m} y_j \cdot I_{\psi_j}.$$ Since $\langle,\rangle_i$
is claimed to be bilinear, then

$$
\langle x,y\rangle_i=\langle x_o\cdot \mathbf{1} + \sum_{i=1}^{k}
x_i \cdot I_{\varphi_i},y_o\cdot \mathbf{1} + \sum_{j=1}^{m} y_j
\cdot I_{\psi_j}\rangle_i = x_o\cdot y_o
\langle\mathbf{1},\mathbf{1}\rangle_i +$$

$$\sum_{j=1}^{m}x_oy_j\langle\mathbf{1},I_{\psi_j}\rangle_i +
\sum_{i=1}^{k}\langle\mathbf{1},I_{\varphi_i}\rangle_i+
\sum_{i=1}^{k} \sum_{j=1}^{m} x_iy_j\langle I_{\varphi_i},
I_{\psi_j}\rangle_i.$$

 Setting ${\langle\bf
1,\bf 1 \rangle}_i=k_{oo}$,
 $\langle{\bf 1},I_\varphi\rangle_i = k_{o,\varphi}$,
 $\langle I_\varphi,I_\psi\rangle_i=k_{\varphi,\psi}$ where $k_{oo},
 k_{o,\varphi}, k_{\varphi,\psi}$ are arbitrarily chosen real
 numbers, one always obtains a bilinear form, but not necessarily positively defined.
  Let us set
\begin{equation}
k_{oo}= \langle\mathbf{1},\mathbf{1}\rangle_i=1,
\end{equation}

\begin{equation}
k_{o,\varphi}=
\langle\mathbf{1},I_{\varphi}\rangle_i=\frac{2}{\pi},
\end{equation}
and
\begin{equation}
k_{\varphi,\psi}=\langle
I_{\varphi},I_{\psi}\rangle_i=\frac{4}{{\pi}^2}(2-\frac{\pi}{2}\cdot
\sin |\varphi_i - \psi_j|).
\end{equation}

Now we set

\begin{definition}\label{ilskaldlaciagow}
For the vectors $x,y \in \mathcal{F}$ defined as above we set

\begin{equation}\label{formuladlacg}
 \langle x,y\rangle_i=x_oy_o + \frac{2}{\pi}\sum_{i=1}^{k}x_iy_o +
\frac{2}{\pi}\sum_{j=1}^{m}x_oy_j + \frac{4}{\pi^2} \sum_{i=1}^{k}
\sum_{j=1}^{m}(2-\frac{\pi}{2}\cdot \sin |\varphi_i -
\psi_j|)x_iy_j.
\end{equation}

\end{definition}

The coefficients in the formula (\ref{formuladlacg}) are chosen in
such a way, that one may easily check directly the reproducing
property of the form $\langle,\rangle_i$. In other words it is
easy to verify, that $\mathcal{F}$ equipped with the scalar
product given by the (\ref{formuladlacg}) is identical (isometric)
with the subspace of the space $X_{\mathcal\mathcal{S}}$ spanned
by the unit disc and all diangles. In consequence
$\langle,\rangle_i$ given by (\ref{formuladlacg}) is really an
inner product (is positively defined) and the completion of
$\mathcal{F}$ equals $X_{\mathcal{S}}$.

Setting in (\ref{formuladlacg}) $x=y$ we obtain the following
formula for the (isoperimetric) norm in $\mathcal{F}$. Namely
\begin{equation}\label{formuladlacg2}
||x||_i^2 = x_o^2 + \frac{4}{\pi}\sum_{i=0}^{k}x_ox_i +
\frac{4}{\pi^2}\sum_{i=1}^{k}\sum_{j=1}^{k}(2-\frac{\pi}{2}
\sin|\varphi_i-\varphi_j|)x_ix_j.
\end{equation}
 This formula may be rewritten in
the form

\begin{equation}\label{formuladlacg3}
||x||_i^2= \frac{4}{\pi^2}\cdot
[(\frac{\pi}{2}x_o+\sum_{i=1}^{k}x_i)^2 -
(\sum_{i=1}^{k}x_i)^2+\sum_{i=1}^{k}
\sum_{j=1}^{k}(2-\frac{\pi}{2}\sin|\varphi_i-\varphi_j|)x_ix_j],
\end{equation}

or in the form
\begin{equation}\label{formuladlacg4}
 ||x_||_i^2= \frac{4}{\pi^2}\cdot
[(\frac{\pi}{2}x_o+\sum_{i=1}^{k}x_i)^2 - (\sum_{i=1}^{k}x_i)^2+
  2\sum_{i,j=1}^{k}x_ix_j -
  \frac{\pi}{2}\sum_{i=1}^{k}\sum_{j=1}^{k}\sin|\varphi_i-\varphi_j|)x_ix_j].
\end{equation}

Now, since $||x||^2\geq 0$ then we obtain the following
inequality:

\begin{equation}\label{formuladlacg5}
(\frac{\pi}{2}x_o+\sum_{i=1}^{k}x_i)^2 +
2\sum_{i,j=1}^{k}x_ix_j\geq (\sum_{i=1}^{k}x_i)^2 +
\frac{\pi}{2}\sum_{i=1}^{k}\sum_{j=1}^{k}\sin|\varphi_i-\varphi_j|)x_ix_j.
\end{equation}

This inequality may be considered as {\it the isoperimetric
inequality for sequences}. If we put $x_0=0$ we obtain the form

\begin{equation}\label{formuladlacg6}
 2\sum_{i,j=1}^{k}x_ix_j\geq
\frac{\pi}{2}\sum_{i=1}^{k}\sum_{j=1}^{k}(\sin|\varphi_i-\varphi_j|)x_ix_j,
\end{equation}
or, finally,  the form

\begin{equation}\label{formuladlacg7}
\sum_{i=1}^{k}\sum_{j=1}^{k}(\sin|\varphi_i-\varphi_j|)x_ix_j\leq
\frac{4}{\pi}(\sum_{i,j=1}^{k}x_ix_j),(=\frac{4}{\pi}(\sum_{i=1}^{k}x_i)^2)).
\end{equation}

Let us remark, that in our interpretation, the sum of diangles
$W=\sum_{i=1}^{k}x_iI_{\varphi_i}$ represents (for non-negatives
$x_i$) a polygon with sides $I_{\varphi_i}$. Since the perimeter
of a diangle is $o(I_\varphi)= 4x_i$ then the number
$(\sum_{i=1}^{k}x_i)^2$ equals $\frac{1}{16}o^2(W)$, so, for
polygon $W$ the last inequality gives

\begin{equation}\label{formuladlacg7}
\sum_{i=1}^{k}\sum_{j=1}^{k}\sin|\varphi_i-\varphi_j|)x_ix_j\leq
\frac{1}{4\pi} o^2(W).
\end{equation}

 It is also not hard to check, that the area of the polygon
$W=\sum_{i=1}^{k}x_iI_{\varphi_i}$  is given by

\begin{equation}\label{formuladlacg8}
area(W) = m(W) =
\sum_{i=1}^{k}\sum_{j=1}^{k}\sin|\varphi_i-\varphi_j|)x_ix_j.
\end{equation}

Hence the classical isoperimetric inequality for sequences says
simply, that the isoperimetric inequality is valid for polygons
(with positive sides). Explicitely (for $W$ as above)
\begin{equation}
\sum_{i=1}^{k}\sum_{j=1}^{k}\sin|\varphi_i-\varphi_j|)x_ix_j =
m(W)\leq \frac{1}{4\pi}o^2(W)= \frac{4}{\pi}(\sum_{i=1}^{k}x_i)^2.
\end{equation}
Let us remark, that as we have proved, the last inequality is
valid for all and  not only for non-negative sequences $(x_i)$.


\vspace{3mm}

\begin{thebibliography}{99}\footnotesize


\bibitem{Aro}
Aronszajn,N.: Theory of Reproducing Kernels,Transactions of the
American Mathematical Society, Vol. 68, 337-404 (1950)

\bibitem{Ber}

Berlinet,A.,Thomas-Agnan,C: Reproducing Kernel Hilbert Spaces in
Probability and Statistics, Springer Science (2004)


\bibitem{Har}

Hardy, G. H.,  Littlewood, J. E., and Pólya, G.: Inequalities, 2nd
ed. Cambridge, England: Cambridge University Press, 1988.



\bibitem{Pau}
Paulsen,Vern,I.: An Introduction to the Theory of Reproduction
Kernel Hilbert Spaces, Dept.of Math., University of Huston, Texas;
https://www.math.uh.edu/~vern/rkhs.pdf


\bibitem{Rad}
 Radstr\"om, H.:
   An Embedding Theorem for Spaces of Convex Sets,
   Proceedings of the American Mathematical Society, Vol.3,
   165-169 (1952)


\bibitem{Szaf}
Szafraniec,F.H.: The Reproducing Kernel Property and Its Space:
The Basics; Springer (2015)

\bibitem{Tre}
Treibergs,A.: Inequalities that Imply the Isoperimetric
Inequality, University of Utah, (2002);
www.math.utah.edu/~treiberg/isoperim/isop.pdf

\bibitem{Tut}
Tutaj,E.: An Example of a Reproducing Kernel Hilbert Space,
Complex Anal. Oper. Theory (2018)
https://doi.org/10.1007/s11785-018-0761-1






\end{thebibliography}
\end{document}